\documentclass[12pt,leqno]{amsart}
\usepackage{amsmath,amssymb,amsfonts}
\usepackage{eucal}
\usepackage{graphicx,psfrag}

\newcounter{outlineonecounter}
\newenvironment{outlineone} {\begin{list} {\arabic{outlineonecounter}.} 
  {\usecounter{outlineonecounter}
  \setlength{\labelwidth}{2em} \setlength{\leftmargin}{2em} 
  }}{\end{list}}
\newcounter{outlineacounter}
\newenvironment{outlinea} {\begin{list} {\alph{outlineacounter}.}  
  {\usecounter{outlineacounter}
  }}{\end{list}}
\newcounter{outlineicounter}

\newtheorem{lem}{Lemma}[section]
\newtheorem{cor}[lem]{Corollary}
\newtheorem{prop}[lem]{Proposition}
\newtheorem{thm}[lem]{Theorem}
  \newtheorem{exama}[section]{Example}

\numberwithin{equation}{section}
\numberwithin{figure}{section}
\numberwithin{table}{section}

\renewcommand\phi{\varphi}                 
\renewcommand\epsilon{\varepsilon}

\newcommand \inv{^{-1}}
\newcommand \transpose{^{\text{\rm T}}}
\newcommand \Diag{\operatorname{Diag}}

\newcommand\Eta{{\mathrm H}}

\newcommand\bb{{\mathbf b}}
\newcommand\bj{{\mathbf j}}
\newcommand\bk{{\mathbf k}}
\newcommand\ba{{\mathbf a}}
\newcommand\bn{{\mathbf n}}
\newcommand\bo{{\mathbf o}}
\newcommand\bp{{\mathbf p}}

\newcommand\cB{\mathcal{B}}
\newcommand\cC{\mathcal{C}}

\newcommand\bbR{\mathbb{R}}

\begin{document}

\begin{center}
{\Large\sc Matrices in the Theory of\\[10pt] Signed Simple Graphs}\\[15pt]
\emph{Thomas Zaslavsky\\
Department of Mathematical Sciences\\
Binghamton University (SUNY)\\
Binghamton, NY 13902-6000, U.S.A.}\\[15pt]
\today\\[15pt]
\end{center}

\small
\emph{Abstract.}  I discuss the work of many authors on various matrices used to study signed graphs, concentrating on adjacency and incidence matrices and the closely related topics of Kirchhoff (`Laplacian') matrices, line graphs, and very strong regularity.
\normalsize
\\[10pt]

\section*{Introduction}

This article is a survey of the uses of matrices in the theory of simple graphs with signed edges.  A great many authors have contributed ideas and results to this field; but amongst them all I have felt to be exceptionally inspiring and important the relevant works of J.J.~Seidel and G.R.~Vijayakumar.

A \emph{signed simple graph} is a graph, without loops or parallel edges, in which every edge has been declared positive or negative.  Such a signed graph is illustrated in Figure \ref{F:sgmulti}(a).  For many purposes the most significant thing about a signed graph is not the actual edge signs, but the sign of each circle (or `cycle' or `circuit'), which is the product of the signs of its edges.  This fact is manifested in simple operations on the matrices I will present.

I treat three kinds of matrix of a signed graph, all of them direct generalisations of familiar matrices from ordinary, unsigned graph theory.  

The first is the adjacency matrix.  The adjacency matrix of an ordinary graph has $1$ for adjacent vertices; that of a signed graph has $+1$ or $-1$, depending on the sign of the connecting edge.  The adjacency matrix leads to questions about eigenvalues and strong regularity.

The second matrix is the vertex-edge incidence matrix.  There are two kinds of incidence matrix of an unsigned graph.  The unoriented incidence matrix has two $1$'s in each column, corresponding to the endpoints of the edge whose column it is.  The oriented incidence matrix has a $+1$ and a $-1$ in each column.  For a signed graph, there are both kinds of column, the former corresponding to a negative edge and the latter to a positive edge.  

Finally, there is the Kirchhoff matrix.\footnote{The Kirchhoff matrix is sometimes called the  `Laplacian', but other matrices are also called `Laplacian'.  I adopt the unambiguous name.}  This is the adjacency matrix with signs reversed, and with the degrees of the vertices inserted in the diagonal.  The Kirchhoff matrix equals the incidence matrix times its transpose.   If we multiply in the other order, the transpose times the incidence matrix, we get the adjacency matrix of the line graph, but with $2$'s in the diagonal.  

All this generalises ordinary graph theory.  Indeed, much of graph theory generalises to signed graphs, while much---though certainly not all---signed graph theory consists of generalising facts about unsigned graphs.

As this article is expository I will give only elementary proofs, to illustrate the ideas.

\section{Fundamentals of Signed Graphs}

\subsection{Definitions about Vectors and Matrices}\label{vdef}

A vector of all $1$'s is denoted by $\bj$.  The matrix $J$ consists of all $1$'s.

\subsection{Definitions about Graphs}\label{gdef}

A graph is $\Gamma = (V,E)$ with vertex set $V$ and edge set $E$.  All graphs will be undirected, finite, and \emph{simple}---without loops or multiple edges---except where explicitly stated otherwise (though most of what I say works well when loops and multiple edges are allowed).  
The \emph{order} of the graph is $n := |V|$.  $c(\Gamma)$ is the number of connected components of $\Gamma$.  $\Gamma^c$ is the complement of $\Gamma$.

An edge with endpoints $v,w$ may be written $vw$ or $e_{vw}$.  An edge with endpoints $v_i,v_j$ may also be written $e_{ij}$.

The \emph{degree} $d_\Gamma(v)$ of a vertex $v$ in a graph $\Gamma$ is the number of edges incident with $v$.  If every vertex has the same degree $k$, we say $\Gamma$ is \emph{regular} of degree $k$, or briefly, \emph{$k$-regular}.

A \emph{walk} in $\Gamma$ is a sequence $W = e_{01}e_{12}\cdots e_{l-1,l}$ of edges, where the second endpoint $v_i$ of $e_{i-1,i}$ is the first endpoint of $e_{i,i+1}$.  Its \emph{length} is $l$.  Vertices and edges may be repeated.  
A \emph{path} is a walk with no repeated vertices or edges.  A \emph{closed path} is a walk of positive length in which $v_0 = v_l$, but having no other repeated vertices or edges.

\subsubsection*{Important Subgraphs.}  
A subgraph of $\Gamma$ is \emph{spanning} if it contains all the vertices of $\Gamma$.  A \emph{circle} (or circuit, cycle, polygon) is the graph of a closed path; that is, it is a $2$-regular connected subgraph.  A \emph{theta graph} consists of three internally disjoint paths joining two vertices.  A \emph{pseudoforest} is a graph in which every component is a tree or a 1-tree (a tree with one extra edge forming a circle).  A \emph{block} of $\Gamma$ is a maximal 2-connected subgraph, or an isthmus or an isolated vertex.  A \emph{cutset} is the set of edges between a vertex subset and its complement, except that the empty edge set is not considered a cutset.

\subsubsection*{Adjacency Matrix.}
The \emph{adjacency matrix} of $\Gamma$ is the $n \times n$ matrix $A(\Gamma)$ in which $a_{ij} = 1$ if $v_iv_j$ is an edge and $0$ if not.  The \emph{Seidel adjacency matrix} of $\Gamma$ is the $n \times n$ matrix $S(\Gamma)$ in which $s_{ij} = 0$ if $i=j$ and otherwise is $-1$ if $v_iv_j$ is an edge, $+1$ if it is not.  Seidel introduced this matrix in \cite{EPS}; his very successful use of it in many papers led to his name's becoming firmly attached.  I will have much to say about the Seidel matrix in terms of signed graphs.

\subsection{Definitions about Signed Graphs}\label{sgdef} 

A \emph{signed graph} $\Sigma$ is a pair $(|\Sigma|,\sigma)$ where $|\Sigma| = (V,E)$ is a graph, called the \emph{underlying graph}, and $\sigma: E \to \{+1,-1\}$ is the \emph{sign function} or \emph{signature}.  The sign group $\{+1,-1\}$ can also be written $\{+,-\}$; I shall treat the two notations as equivalent.\footnote{The important thing is that the signs form a multiplicative group of two elements that acts on numbers.  For matrix theory the notation is simpler if they are themselves numbers.}  
Often, we write $\Sigma = (\Gamma,\sigma)$ to mean that the underlying graph is $\Gamma$.  $E^+$ and $E^-$ are the sets of positive and negative edges.  The \emph{positive} and \emph{negative subgraphs},  
$$
\Sigma^+ = (V, E^+) \text{ and } \Sigma^- = (V, E^-) ,  
$$ 
are unsigned graphs.

A signed graph is \emph{simply signed} if it has no loops\footnote{Usually, negative loops would be allowed; but that is not so suitable to the matrix theory.} and no parallel edges with the same sign; but it may have two edges, one positive and one negative, joining a pair of vertices.  We shall be concerned mostly with signed simple graphs (those with no parallel edges) but occasionally simply signed graphs, and even loops, will play a role.\\

\begin{figure}[hb]
\includegraphics[scale=0.8]{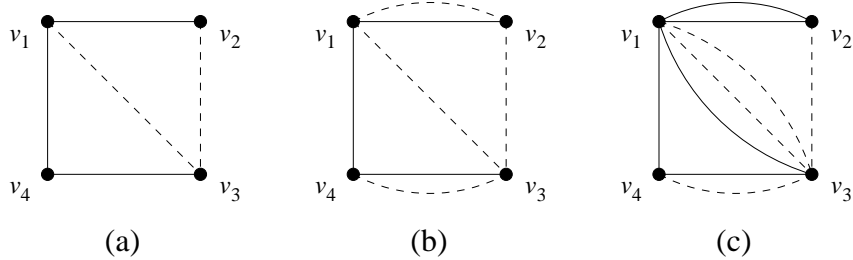}
\caption{A signed simple graph $\Sigma_4$ (a), a simply signed multigraph (b), and a signed multigraph that is not simply signed (c).}
\label{F:sgmulti}
\end{figure}

Signed graphs $\Sigma$ and $\Sigma'$ are \emph{isomorphic} if there is a graph isomorphism $f: |\Sigma| \to |\Sigma'|$ that preserves edge signs.

Define $\Sigma$ to be \emph{regular} if both $\Sigma^+$ and $\Sigma^-$ are regular graphs.

\subsection{Examples}

\begin{outlineone}

\item We say $\Sigma$ is \emph{homogeneous} if all edges have the same sign, and \emph{heterogeneous} otherwise.  (I got this handy terminology from M.~Acharya and co-authors.)  It is \emph{all positive} or \emph{all negative} if all edges are positive or negative, respectively.  $+\Gamma$ denotes $\Gamma$ with all positive signs.  $-\Gamma$ is $\Gamma$ with all negative signs.

\smallskip
\item $K_{\Gamma_0}$ denotes a complete graph $K_n$ with vertex set $V = V({\Gamma_0})$, whose edges are negative if they belong to ${\Gamma_0}$ and positive otherwise.  That is, $(K_{\Gamma_0})^- = {\Gamma_0}$ and $(K_{\Gamma_0})^+ = {\Gamma_0}^c$.
 
\end{outlineone}

\subsection{Walks, Circles, and their Signs}\label{sgw}

The sign of a walk $W = e_1e_2\cdots e_l$ is the product of its edge signs: $$\sigma(W) := \sigma(e_1)\sigma(e_2)\cdots\sigma(e_l).$$  Thus, a walk is either positive or negative, depending on whether it has an even or odd number of negative edges, counted with their multiplicity in $W$ if there are repeated edges.

A circle, being the graph of a closed path, has a definite sign, either positive or negative.  
The class of positive, or negative, circles in $\Sigma$ is denoted by $\cC^+(\Sigma)$, or $\cC^-(\Sigma)$, respectively.  
One can characterise the class of negative circles by the following property.  A \emph{theta graph} is the union of three paths joining the same two vertices, but otherwise disjoint.

\begin{prop}[cf.\ \cite{CSG}]\label{P:circles}
A class $\cB$ of circles in a graph $\Gamma$ is the class of negative circles of a signed graph $(\Gamma,\sigma)$ if and only if every theta subgraph contains an even number of circles in $\cB$.
\end{prop}

A negative circle of length $2$---which consists of one positive edge and one negative edge and can only occur in a signed multigraph---is called a \emph{negative digon}.

\subsection{Balance}

A signed graph $\Sigma$, or a subgraph or edge set, is called \emph{balanced} if every circle in it is positive.  
$b(\Sigma)$ is the number of connected components of $\Sigma$ that are balanced.  For $S \subseteq E$, $b(S)$ is the number of balanced components of $(V,S)$.  

A circle is balanced iff it is positive.  A walk is called balanced when its underlying graph is balanced; thus, a positive walk may be balanced or unbalanced, and the same holds for a negative walk.

It is easy to see that $\Sigma$ is balanced if and only if every block is balanced.

\begin{thm}[Harary's Balance Theorem \cite{NB}]\label{T:balance}
A necessary and sufficient condition for $\Sigma$ to be balanced is that there be a bipartition of\/ $V$ into $X$ and $Y$ such that an edge is negative precisely when it has one endpoint in $X$ and one in $Y$.  ($X$ or $Y$ may be empty.)
\end{thm}

In other words, $\Sigma$ is balanced just when $E^-$ is empty or a cutset.  When $\Sigma$ is balanced, a \emph{Harary bipartition} is any bipartition $\{X,Y\}$ as in the theorem.  It is unique if and only if $\Sigma$ is connected.  It was not so easy to prove this theorem at the time, but with switching it becomes simple; see later.

\begin{cor}\label{C:bipartite}
$-\Gamma$ is balanced iff\/ $\Gamma$ is bipartite.
\end{cor}

This means balance is a kind of generalisation of biparticity.  That turns out to be a valid insight in a number of ways; e.g., in regard to the unoriented incidence matrix of a graph (Example \ref{ix}.\ref{X:unoriented}).

We call $\Sigma$ \emph{antibalanced} if $-\Sigma$ is balanced; equivalently, if all even circles are positive and all odd circles are negative.  The quick way to get an antibalanced signed graph is to give negative signs to all edges of a graph.  Corresponding to Harary's Balance Theorem is the following result:

\begin{cor}\label{C:antibalance}
A necessary and sufficient condition for $\Sigma$ to be antibalanced is that there be a bipartition of\/ $V$ into $X$ and $Y$ such that an edge is positive precisely when it has one endpoint in $X$ and one in $Y$.  ($X$ or $Y$ may be empty.)
\end{cor}

Thus, $\Sigma$ is antibalanced iff $E^+$ is empty or a cutset.  Correspondingly to Corollary \ref{C:bipartite}, $+\Gamma$ is antibalanced iff $\Gamma$ is bipartite.

\subsection{Switching}\label{sw}

\emph{Switching} $\Sigma$ means reversing the signs of all edges between a vertex set $X$ and its complement.  $X$ may be empty.  We say $X$ is \emph{switched} in $\Sigma$.  The switched graph is written $\Sigma^X$.   \emph{Vertex switching} means switching a single vertex.  Switching a set $X$ has the same effect as switching all the vertices in $X$, one after another.

Another version of switching, which is equivalent to the preceding and is very useful, is in terms of a function $\theta: V \to \{+,-\}$, called a \emph{switching function}.  \emph{Switching $\Sigma$ by $\theta$} means changing $\sigma$ to $\sigma^\theta$ defined by $$\sigma^\theta(vw) := \theta(v) \sigma(vw) \theta(w).$$  The switched graph is written $\Sigma^\theta := (|\Sigma|, \sigma^\theta).$ 

If $\Sigma$ can be switched to become $\Sigma'$, we say $\Sigma$ and $\Sigma'$ are \emph{switching equivalent}.  Switching equivalence is an equivalence relation on signatures of a fixed graph.  An equivalence class is called a \emph{switching class}.  $[\Sigma]$ denotes the switching class of $\Sigma$.

If $\Sigma'$ is isomorphic to a switching of $\Sigma$, we say $\Sigma$ and $\Sigma'$ are \emph{switching isomorphic}.  Switching isomorphism is an equivalence relation on all signed graphs.  Often in the literature switching isomorphism is not distinguished from switching equivalence, but I prefer to separate the two concepts.

\begin{lem}[Switching Lemma {\cite[Corollary 3.3]{SG}}]\label{L:switching}
$\Sigma$ is balanced if and only if it switches to an all-positive signature, and it is antibalanced if and only if it switches to an all-negative signature.
\end{lem}

\begin{proof}
To prove the first statement one can assume $\Sigma$ is connected.  Take a spanning tree, rooted at any vertex $v$, and switch $\Sigma$ so the tree is all positive.  The switching function for this is $\theta(w) := \sigma(T_{vw})$, where $T_{vw}$ is the unique $vw$-path in $T$.  $\Sigma$ is balanced if and only if there are no remaining negative edges.  

The second statement follows by a similar proof, or by negation from the first part.
\end{proof}

A useful way to think of the Switching Lemma without actually switching is that $\Sigma$ is balanced iff there is a function $\mu: V \to \{+,-\}$ such that $\sigma(e_{vw}) = \mu(v)\mu(w)$.  This amounts to saying that $\sigma$ has a potential function, i.e., $\mu$.

Properties preserved by switching are the signs of circles, and balance or imbalance of $\Sigma$ and of any subgraph; also deletion sets and negation sets.  

The proof technique of the Switching Lemma yields a valuable insight into equivalence of signed graphs.

\begin{thm}[Switching Equivalence {\cite{Soz}, \cite[Proposition 3.2]{SG}}]\label{T:swequiv}
Two signed graphs with the same underlying graph are switching equivalent if and only if they have the same class of positive circles.
\end{thm}

\begin{proof}
As in the previous proof we may assume the graphs, $\Sigma_1$ and $\Sigma_2$, are connected.  Switch both signed graphs so a fixed spanning tree $T$ is all positive.  Call the switched graphs $\Sigma_1'$ and $\Sigma_2'$.  

If $\Sigma_1$ and $\Sigma_2$ are switching equivalent, so are $\Sigma_1'$ and $\Sigma_2'$, but as they agree on a spanning tree, one can only be switched to the other by no switching at all or by switching every vertex.  Thus, $\Sigma_1' = \Sigma_2'$, whence $\cC^+(\Sigma_1) = \cC^+(\Sigma_1') = \cC^+(\Sigma_2') = \cC^+(\Sigma_2)$.

If $\cC^+(\Sigma_1) = \cC^+(\Sigma_2)$, then after switching $\cC^+(\Sigma_1') = \cC^+(\Sigma_2')$.  Now, with an all-positive spanning tree $T$, the only possible difference between $\Sigma_1'$ and $\Sigma_2'$ is in the signs of the non-tree edges.  But the sign of an edge $e \notin T$ is the sign of the unique circle in $T \cup \{e\}$, which is the same in $\Sigma_1'$ and $\Sigma_2'$.  Therefore, $\Sigma_1' = \Sigma_2'$.  As $\Sigma_1$ and $\Sigma_2$ switch to the same signed graph $\Sigma_1'$, they are switching equivalent.
\end{proof}

\begin{cor}[Switching Isomorphism \cite{Soz}]\label{T:swisom}
Two signed graphs are switching isomorphic if and only if there is an isomorphism of underlying graphs that preserves the signs of circles.
\end{cor}

With switching I can give short proofs of such results as Harary's fundamental theorem.

\begin{proof}[Proof of Harary's Balance Theorem] 
If there is such a bipartition, every circle has an even number of negative edges, so $\Sigma$ is balanced.  

If $\Sigma$ is balanced, switch it to be all positive.  Letting $X$ be the set of switched vertices, the bipartition is $\{X,V\setminus X\}$. 
\end{proof}

\subsection{History}

Signed graphs and balance were invented by Harary by 1953 \cite{NB} to treat a question in social psychology \cite{CH}.  They have since been reinvented over and over in many contexts in physics, geometry, economics, and more---thus showing that they are a natural concept---but it was Harary who first had the idea of labelling with the 2-element group by putting signs on the edges and multiplying them.  
Harary also introduced antibalance, in \cite{SD}.

It is remarkable that, years before Harary, K\"onig \cite[Section X.3]{Konig} had the idea of a graph with a distinguished subset of edges, in a way we now recognise as equivalent to signed graphs, and even proved Harary's Balance Theorem and defined switching in the form of taking the set sum of a cutset with the set of negative edges.  Despite all this, he failed to notice the value of edge labels that one can multiply, which I regard as the crucial step in the invention of signed graphs.

The first to think of switching as an operation on signed graphs were the social psychologists Abelson and Rosenberg \cite{PsL}.  However, their formalism for switching, in terms of a Hadamard product operation on their adjacency matrix $R$ (see Section \ref{aar}), was awkward and ungraphical.  I arrived at switching of signed graphs as an adaptation and generalisation of graph switching, introduced by Seidel and applied with great effect in many papers (cf.\ \cite{EPS, STG}).  Switching, like signed graphs themselves, has been reinvented several times, with varying names and notation.

The observation following the proof of the Switching Lemma was published, independently of other work on switching, by Sampathkumar \cite{PSLS}.

\section{Adjacency Matrices}

\subsection{Definition}\label{adef}

The adjacency matrix $A = A(\Sigma)$ is an $n \times n$ matrix in which $a_{ij} = \sigma(v_iv_j)$ (the sign of the edge $v_iv_j$) if $v_i$ and $v_j$ are adjacent, and $0$ if they are not.  Thus $A$ is a symmetric matrix with entries $0, \pm1$ and zero diagonal, and conversely, any such matrix is the adjacency matrix of a signed simple graph.  The absolute value matrix, $|A(\Sigma)|$, equals $A(|\Sigma|)$, the adjacency matrix of the underlying unsigned graph.

When $\Sigma$ has multiple edges, the $(i,j)$ entry of $A$ is the sum of the signs of all $v_iv_j$ edges.  Positive and negative edges cancel each other.  This will become important to us in the treatment of line graphs.\\

\begin{figure}[ht]
\parbox{11em}{
\begin{center}
$
\begin{pmatrix}
0 & 1 & -1 & 1 \\
1 & 0 & -1 & 0 \\
-1 & -1 & 0 & 1 \\
1 & 0 & 1 & 0
\end{pmatrix}
$
\vspace{.4cm}

(a) $A(\Sigma_4)$
\end{center}
}
\quad
\parbox{11em}{
\begin{center}
$
\begin{pmatrix}
0 & 0 & -1 & 1 \\
0 & 0 & -1 & 0 \\
-1 & -1 & 0 & 0 \\
1 & 0 & 0 & 0
\end{pmatrix}
$
\vspace{.4cm}

(b)
\end{center}
}
\\[10pt]
\parbox{11em}{
\begin{center}
$
\begin{pmatrix}
0 & 2 & -1 & 1 \\
2 & 0 & -1 & 0 \\
-1 & -1 & 0 & 0 \\
1 & 0 & 0 & 0
\end{pmatrix}
$
\vspace{.4cm}

(c)
\end{center}
}
\caption{The adjacency matrices of the signed graphs in Figure \ref{F:sgmulti}.  Note the cancellations due to simultaneous adjacencies with opposite signs in (b, c) and the $2$ for double adjacency in (c).}
\label{F:adjacency}
\end{figure}
%

\subsection{Walks and Neighbors}\label{awalk}

Powers of $A$ count walks in a signed way.  Let $w^+_{ij}(l)$ be the number of positive walks of length $l$ from $v_i$ to $v_j$ (that is, the sign product of the edges in $W$ is positive), and let $w^-_{ij}(l)$ be the number of negative walks.  

\begin{thm}\label{T:walkscount}
The $(i,j)$ entry of $A^l$ is $w^+_{ij}(l) - w^-_{ij}(l)$.
\end{thm}

I omit the proof, which is not difficult.

A noteworthy special case is the square of $A$.  Let $p^{+}_{ij}$ denote the number of common positive neighbors of distinct vertices $v_i$ and $v_j$, $p^{-}_{ij}$ the number of their common negative neighbors, and $p^\pm_{ij}$ the number of neighbors that are positive neighbors of one and negative neighbors of the other.

\begin{cor}\label{C:nbrs}
In $A(\Sigma)^2$ the $(i,i)$ entry is $d_{|\Sigma|}(v_i)$ and the $(i,j)$ entry, for $i\neq j$, is $p^{+}_{ij} + p^{-}_{ij} - p^\pm_{ij}.$
\end{cor}

\begin{proof}
The diagonal entry is the number of walks $v_iv_jv_i$, since every such walk has the positive sign $\sigma(v_iv_j)^2$.  This number equals the number of neighbors of $v_i$.  As for the off-diagonal entry, $p^{+}_{ij} + p^{-}_{ij} = w^+_{ij}(2)$ and $p^\pm_{ij} = w^-_{ij}(2)$.
\end{proof}

\subsection{Examples}\label{ax}

\begin{outlineone}

\item If $\Sigma$ is complete, i.e., $|\Sigma| = K_n$, then $A(\Sigma)$ is the Seidel adjacency matrix $S(\Sigma^-)$ of the negative subgraph $\Sigma^-$.  That means the Seidel matrix of a graph is the adjacency matrix of a signed complete graph.  This fact inspired my work on adjacency matrices of signed graphs.

\smallskip
\item If $\Sigma$ is bipartite, so that $|\Sigma| \subseteq K_{r,s}$, then $A(\Sigma) = \begin{pmatrix} O & B \\ B\transpose & O \end{pmatrix}$ where $B$ is an $r \times s$ matrix of $0$'s, $+1$'s, and $-1$'s.  If $\Sigma$ is complete bipartite, $B$ has no $0$'s.
\label{X:axbi}

\end{outlineone}

\subsection{Switching}\label{asw}

Switching has a simple effect on $A(\Sigma)$.  Given $X \subseteq V$, switching by $X$ negates both the rows and columns of vertices in $X$.  Given a function $\theta: V \to \{+1,-1\}$, switching by $\theta$ negates both the rows and columns of vertices with $\theta(v_i) = -1$.
In matrix terms, let $\Diag(\theta)$ be the diagonal $(0,\pm1)$-matrix with $\theta(v_i)$ in the $i$th diagonal position.  Then switching by $\theta$ conjugates $A(\Sigma)$ by $\Diag(\theta)$; that is, 
$$
A(\Sigma^\theta) = \Diag(\theta)\inv A(\Sigma) \Diag(\theta).
$$
(Note that $\Diag(\theta)\inv = \Diag(\theta)$; I inserted the inversion in order to show that this is truly conjugation.)  We may conclude that:

\begin{prop}\label{P:swdiag}
Signed graphs\/ $\Sigma_1$ and\/ $\Sigma_2$ on the same vertex set are switching equivalent if and only if their adjacency matrices satisfy $A(\Sigma_2) = D \inv A(\Sigma_1) D$ for some diagonal $(0,\pm1)$-matrix $D$ whose diagonal has no zeros.
\end{prop}

We can strengthen this to a criterion for switching \emph{isomorphism} (if not one that is computationally practical).  By `similarly rearranging' the rows and columns of a square matrix $A$, I mean that one applies the same permutation to the rows and to the columns.

\begin{cor}\label{C:swisomdiag}
Signed graphs\/ $\Sigma_1$ and\/ $\Sigma_2$ are switching isomorphic if and only if the rows and columns of $A(\Sigma_1)$ can be similarly rearranged so that $A(\Sigma_2) = D\inv A(\Sigma_1) D$ for some diagonal $(0,\pm1)$-matrix $D$ whose diagonal has no zeros.
\end{cor}

I omit the simple verifications.  

Similarly rearranging the rows and columns is represented with matrices as conjugation by a permutation matrix $P$, i.e., $A$ is changed to $P\inv A P$.  Thus, in strictly matrix language, $A(\Sigma_2) = (PD)\inv A(\Sigma_1) (PD)$.

If $\Sigma$ is bipartite, as in Example \ref{ax}.\ref{X:axbi}, with bipartition $V = V_1 \cup V_2$, we can examine the effect of switching $X$ on $B$.  Say the rows of $B$ are indexed by $V_1$ and the columns by $V_2$.  Switching negates each row of $B$ corresponding to a vertex in $X \cap V_1$ and each column corresponding to a vertex in $X \cap V_2$.

\subsection{Eigenvalues and Eigenvectors}\label{e}

An \emph{eigenvalue of\/ $\Sigma$} is an eigenvalue of its adjacency matrix.  The \emph{spectrum} of $\Sigma$ is the list of its eigenvalues with their multiplicities.  Individual eigenvalues and the spectrum may give information about $\Sigma$.  I know of nothing that has been done on this aside from special cases: signed graphs whose underlying graph is regular, which can be treated through the Kirchhoff matrix (see Section \ref{ke}); those with no eigenvalues greater than $2$ (they are, with a few exceptions, line graphs of signed graphs; see Section \ref{lgae}); and Acharya's matricial criterion for balance.  There is one fundamental result that must be the starting point for all investigations:

\begin{prop}\label{P:swspec}
Switching a signed graph does not change its spectrum.
\end{prop}

\begin{proof}
This is a corollary of Proposition \ref{P:swdiag}, because conjugating a matrix does not change its spectrum.
\end{proof}

\begin{prop}[B.D.\ Acharya \cite{Aeigen}]\label{P:aebalance}
$\Sigma$ is balanced if and only if\/ $A(\Sigma)$ has the same eigenvalues (with multiplicities) as does $A(|\Sigma|)$.
\end{prop}

I omit the proof.  
Acharya's criterion is the more interesting because it cannot be strengthened to an eigenvalue criterion for switching isomorphism.  Although switching-isomorphic signed graphs obviously must have the same eigenvalues, the converse is false, as one can see from the facts that nonisomorphic unsigned graphs can have the same eigenvalues and that $A(+\Gamma) = A(\Gamma)$.

A regular signed graph has $d^\pm(\Sigma)$ as an eigenvalue, associated to the eigenvector $\bj$.  

\begin{prop}\label{P:aevector}
$\Sigma$ is regular if and only if\/ $\bj$ is an eigenvector of both $A(|\Sigma|)$ and $A(\Sigma)$.
\end{prop}

\begin{proof}
The term of $v_i$ in $A(|\Sigma|)\bj$ is $d_{|\Sigma|}(v_i) = d_{\Sigma^+}(v_i) + d_{\Sigma^-}(v_i)$, the sum of degrees in the positive and negative subgraphs.  
The term of $v_i$ in $A(\Sigma)\bj$ is $d_{\Sigma^+}(v_i) - d_{\Sigma^-}(v_i)$, the difference of degrees.  
A necessary and sufficient condition for $\bj$ to be an eigenvector of both is that the sum and difference be independent of $i$; equivalently, that the positive and negative subgraphs be regular; thus, by definition, that $\Sigma$ be regular.
\end{proof}

Although it does not directly provide information about the signed graph itself, from matrix theory we do know something about the relationship between eigenvalues of $\Sigma$ and those of induced subgraphs.

\begin{prop}\label{P:einterlacing}
If $\Sigma_1$ is an induced subgraph of $\Sigma_2$, then the largest eigenvalues satisfy $\lambda^1(\Sigma_1) \leq \lambda^1(\Sigma_2)$.
\end{prop}

\begin{proof}
This is an instance of the general interlacing theorem for eigenvalues of real, symmetric matrices \cite[Theorem 9.1.1]{AGT}.
\end{proof}

McKee and Smyth \cite{MS} start with an eigenvalue question and arrive at signed graphs.  If a symmetric integer matrix $A$ has all its eigenvalues (which are necessarily real) in the interval $[-2,2]$, then it is the adjacency matrix of a signed simple graph except for having possibly nonzero elements on the diagonal.  The signed graphs that can appear are subgraphs of members of three infinite families and a handful of sporadic examples.  Their result is one of several reasons to think that $\pm2$ are especially significant bounds on eigenvalues.  I refer the reader to their long paper for the details.

As I mentioned, the eigenvalue properties of signed graphs are an open field of inquiry.

\subsection{The Abelson--Rosenberg Adjacency Matrix}\label{aar} 

An oddity from early in signed graph theory is the adjacency matrix formulated by the inventive social psychologists Abelson and Rosenberg \cite{PsL}.  Their matrix, $R$, has entries $\bo, \bp, \bn, \ba$ (and these symbols have addition and multiplication operations), where $\bo$ stands for nonadjacency, $\bp$ and $\bn$ are for positive and negative adjacency, and $\ba$ is for vertices that are adjacent by both a positive and a negative edge; and it has $\bp$ on the diagonal instead of $\bo$.  
Abelson and Rosenberg employed their symbols to stand for an `unrelated', `positive', `negative', and `ambivalent' relationship between the vertices, which represented persons in a social group.   
The symbol $\ba$, which can occur with a simply signed graph but not a signed simple graph, was necessary for the psychology.

Harary, Norman, and Cartwright \cite[Theorem 13.8]{HNC} used $R$ to show the existence of walks of given length and sign between two specified vertices.  Their theorem says there is a positive (resp., negative) walk of length $l$ from $v_i$ to $v_j$ iff the $(i,j)$ position in $(R-\bp I)^l$ contains $\bp$ or $\ba$ (resp., $\bn$ or $\ba$).

\subsection{Degrees}\label{adeg}

The degree of a vertex has several generalisations to signed graphs.  The \emph{underlying degree} $d_{|\Sigma|}(v)$ is the degree in the underlying graph, that is, the total number of edges incident with $v$.  
The \emph{positive} or \emph{negative degree}, $d^+_\Sigma(v)$ or $d^-_\sigma(v)$, is the degree in the positive subgraph $\Sigma^+$ or the negative subgraph $\Sigma^-$.  The \emph{net degree} is 
$$
d^\pm_\Sigma(v) := d^+_\Sigma(v)  - d^-_\Sigma(v) .
$$

If $\Sigma$ is regular, we can speak of its \emph{degree} (of any kind): it is the degree of any vertex.  For example, the net degree of $\Sigma$ is $d^\pm(\Sigma) := d^\pm(v)$ for every vertex $v$.  $\Sigma$ is regular if and only if $|\Sigma|$ is regular and $A \bj = r \bj$ for some real number $r$, which necessarily equals $d^\pm(\Sigma).$

\subsection{Very Strong Regularity}\label{vsr}

Seidel discovered that a strongly regular graph ${\Gamma_0}$ has a nice definition in terms of its Seidel adjacency matrix $S=S({\Gamma_0})$, namely, that 
$$
S^2 - tS - kI = p(J-I) \text{ and } S\bj = \rho_0\bj
$$ 
for some constants $t, k, p, \rho_0$ (thus in particular $\bj$ is an eigenvector of $S$).  (See \cite{SRG3, STG}.)
The constants have combinatorial interpretations.  For a start, $k=n-1$, the degree of $K_n$.  Given an edge $vw$, the number of vertices adjacent to exactly one of $v$ and $w$ equals $\frac12 (n-2-p+t)$.  Given a nonadjacent pair $v,w$, that number is $\frac12 (n-2-p-t)$.
If $p \neq 0$, the second defining equation $S\bj = \rho_0\bj$ is superfluous.

Contemplating the fact that $S({\Gamma_0})$ is the same matrix as $A(K_{\Gamma_0})$, I was somehow led to the following definition:
A signed graph is \emph{very strongly regular} \cite{WSR} if its adjacency matrix satisfies 
\begin{equation}\label{E:vsr}
A^2 - tA - kI = p\bar A \text{ and } A\bj = \rho_0\bj
\end{equation}
for some constants $t, k, p, \rho_0$.  Here $\bar A$ is the adjacency matrix of the complement of $|\Sigma|$.  The combinatorial interpretation of these parameters is:
      \begin{outlinea}
      \item $|\Sigma|$ is $k$-regular.
      \item $\rho_0 = d^\pm(\Sigma)$, the net degree of every vertex (hence it is an integer).  Hence $\Sigma$ is regular.
      \item $t = t^+_{ij} - t^-_{ij}$ where $t^+_{ij},\ t^-_{ij}$ are the numbers of positive and negative triangles on an edge $e_{ij}$.
      \item $p = p^+_{ij} - p^-_{ij}$ for any pair $v_i, v_j$ of nonadjacent vertices, where $p^+_{ij},\ p^-_{ij}$ are the numbers of positive and negative length-2 paths joining the vertices.
      \item $t$ and $p$ are independent of the choices of adjacent or nonadjacent vertices.
      \end{outlinea}

The big problem is to, in some sense, classify very strongly regular signed graphs.  I am currently working on this.  Various simplifications are helpful.  For instance, $-\Sigma$, whose signature is $-\sigma$, behaves just like $\Sigma$ except that $t$ and $\rho_0$ are negated.  Thus, one may assume that $t$ is nonnegative, or that $\Sigma$ is not all negative, when it is convenient to do so.  

The most important factor in the classification is which of $p$ and $t$ are 0.
As is usual in such problems, there are strong numerical restrictions on the values of $t,k,p,\rho_0$.  It is interesting that some of the types include kinds of matrices that have already been studied for many years.  I will run down the possibilities.  I write $d^\pm$ for $d^\pm(\Sigma)$.
A \emph{weighing matrix} is a $(0,\pm1)$-matrix $W$ such that $W\transpose W$ is a multiple of $I$ \cite{HCD2W}.
\smallskip
      \begin{enumerate}
      \item[(a)] Homogeneous signed graphs can be assumed (by negation) to be all positive.  $+\Gamma$ is very strongly regular if and only if $\Gamma$ is a strongly regular unsigned graph.  This demonstrates that we have a true generalisation of strong regularity.  
      \end{enumerate}
For the rest of the cases I assume $\Sigma$ is inhomogeneous.
      \begin{enumerate}
      \item[(b)] When $p = t = 0$, the defining equations are $A^2 = kI$ and $A\bj = d^\pm\bj$.  The eigenvalues are $\pm\surd k$, one of which is $d^\pm$.  Consequently, $k$ is a square and $A$ is a symmetric weighing matrix with zero diagonal and in each row $\binom{d^\pm+1}{2}$ entries equal to $+1$ and $\binom{d^\pm}{2}$ entries equal to $-1$.
      \item[(c)] When $p = 0$ but $t \neq 0$,  the defining equations are $A^2 - tA - kI = O$ and $A\bj = d^\pm\bj$.  Solving for the eigenvalues shows there is an integer $s \equiv t \text{ (mod }2)$ such that 
      $$k = \frac{s-t}{2} \cdot \frac{s+t}{2} \quad\text{and }\quad d^\pm = \frac{s+t}{2} \,.$$  
      \item[(d)] When $p \neq 0$, from the eigenvector $\bj$ we deduce that $p = [d^\pm(d^\pm+t) - k] / [n-1-k],$ which significantly constrains the numbers.  This case is too complicated for further description here.  For instance, the eigenvalues of $\Sigma$ depend on those of $|\Sigma|^c$.
      \end{enumerate}

Much more can be said, but not here; see \cite{WSR}.

\subsection{History}

To my knowledge, the first adjacency matrix of a signed graph was that of Abelson and Rosenberg.  The standard adjacency matrix, $A(\Sigma)$, appeared soon after, but I am not sure exactly when and where.  Harary certainly used it early on.

Switching of the standard adjacency matrix is implicit in \cite{EPS} and explicit in a form that is equivalent to switching signed complete graphs in \cite{SRG3} but was not made explicitly and generally signed-graphic until later (I am not sure just when).  
Abelson and Rosenberg, as I mentioned, had the concept of switching but they did not develop it far, and I think that would have been difficult given their mathematical formulation.

\section{Orientation}

\subsection{Bidirected Graphs}\label{bidirected}

In a \emph{bidirected graph}, every edge has an independent orientation at each end.  (This concept is due to Jack Edmonds; cf.\ \cite{EJbi}.)  We think of these in two ways: as an arrow at each end, which may point towards or away from the endpoint, and as a sign $\eta(v,e)$ on the end of $e$ that is incident with $v$, which is $+1$ if the arrow points to the endpoint, $-1$ if the arrow is directed away from the endpoint.

\subsection{Oriented Signed Graphs}\label{osg} 

A bidirected graph is naturally sign\-ed by the formula 
\begin{equation}\label{E:sign}
\sigma(e) = - \eta(v,e) \eta(w,e)
\end{equation}
for an edge $e_{vw}$.  An edge is negative if its arrows both point toward their corresponding endpoints (I call such edges \emph{extraverted}), or both away from their endpoints (I call these \emph{introverted} edges).  An edge is positive if one arrow points at its endpoint while the other is directed away from its endpoint.  Thus, a positive edge is just like an ordinary directed edge.

Conversely, an \emph{orientation} of a signed graph $\Sigma$ is a bidirection $\eta$ of $|\Sigma|$ that satisfies the sign formula \eqref{E:sign}.

\emph{Reorienting} an edge $e_{vw}$ means replacing $\eta(v,e)$ and $\eta(w,e)$ by their negatives.   In terms of arrows, it means reversing the arrows at both endpoints of the edge.  This does not change the sign of the edge.

\begin{figure}[hb]
\includegraphics[scale=0.8]{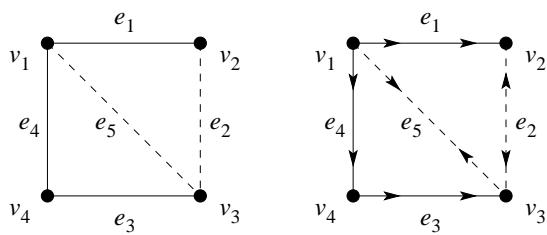}
\caption{The signed graph $\Sigma_4$ and an orientation of it.}
\label{F:sgor}
\end{figure}
%

\subsection{Switching}\label{osw} 

One switches a vertex set $X$ in a bidirected graph by changing the signs of all the edge ends incident with a vertex in $X$.  In terms of a switching function $\theta: V \to \{+,-\}$, one has $\eta^\theta(v,e) := \theta(v)\eta(v,e)$.

Recall that switching the entire vertex set $V$ does not change a signed graph; but it does have an effect on an orientation: it reverses all the arrows.

\section{Incidence Matrices}

\subsection{Definition}\label{idef}

An \emph{incidence matrix} of $\Sigma$ is a $V \times E$ matrix in which the column of edge $e$ has two entries $\pm1$, one in the row of each endpoint of $e$, and 0's elsewhere.  The two nonzero entries must have product equal to $-\sigma(e)$; that is, they are equal if $e$ is negative, but if $e$ is positive, one is $+1$ and the other is $-1$.

The notation I use for an incidence matrix of a signed graph is $\Eta(\Sigma)$ (read `Eta') $= (\eta_{ve})_{v\in V, e\in E}.$
\begin{figure}[hb]
$\Eta(\Sigma_4,\eta) = 
\begin{pmatrix}
-1 & 0 & 0 & -1 & -1 \\
+1 & +1 & 0 & 0 & 0 \\
0 & +1 & +1 & 0 & -1 \\
0 & 0 & -1 & +1 & 0 \\
\end{pmatrix}
$
\caption{The incidence matrix of $\Sigma_4$ corresponding to the orientation $\eta$ in Figure \ref{F:sgor}.}
\label{F:sgincidence}
\end{figure}

The incidence matrix is not unique.  The choice of signs in each column reflects a choice of orientation $\eta$ of $\Sigma$; in fact, the $(v,e)$ entry in $\Eta(\Sigma)$ is equal to 
$$
\begin{cases}
\eta(v,e)	&\text{ if } e \text{ is incident with } v, \\
0	&\text{ if } e \text{ is not incident with } v.
\end{cases}
$$
(Thus, we may define $\eta(v,e) := 0$ in the latter case.  Then $\eta(v,e)=\eta_{ve}$; the only difference between the two is the point of view: the former is $\eta$ regarded as a bidirection and the latter is a matrix entry.)  Conversely, the entries of an incidence matrix of $\Sigma$ determine an orientation $\eta$.

The incidence matrix can be treated as a matrix over any field $\bk$, or indeed any ring, in particular over the ring of integers, or the $2$-element field.

\subsection{Switching and Reorientation}\label{isw}

The effect on $\Eta(\Sigma)$ of switching $X \subseteq V$ is to negate the rows corresponding to the vertices in $X$.  In terms of a switching function $\theta$ and the matrix $\Diag(\theta)$, $\Eta(\Sigma)$ switches to $\Diag(\theta)\Eta(\Sigma)$.  

The effect of reorienting edges is to negate the corresponding columns of $\Eta(\Sigma)$.

\subsection{Rank}\label{irank}

The rank of $\Eta(\Sigma)$ is a basic fact that generalises more widely known but limited results about incidence matrices.

\begin{thm}[{Jeurissen \cite{Jeur}, Zaslavsky \cite[Section 8A]{SG}}]\label{T:rank}
If\/ $\bk$ has characteristic $2$, then $\Eta(\Sigma)$ has rank $n - c(\Sigma)$.  Otherwise, $\Eta(\Sigma)$ has rank $n - b(\Sigma)$.
\end{thm}

\begin{proof}  
As $\Eta(\Sigma)$ can be written in the block form 
$$
\begin{pmatrix} \Eta(\Sigma_1) & O &\cdots & O \\ O & \Eta(\Sigma_2) &\cdots & O \\ &&\cdots& \\ O & O &\cdots & \Eta(\Sigma_k) \end{pmatrix},
$$
where $\Sigma_1, \Sigma_2, \ldots, \Sigma_k$ are the connected components of $\Sigma$, it suffices to prove the rank formula for a connected signed graph.  For $S \subseteq E$, write $\Eta(S)$ for the submatrix of $\Eta(\Sigma)$ that consists of the columns corresponding to $S$.

Let $T$ be a spanning tree of $\Sigma$ and switch so $T$ is all positive.  If $\Sigma$ is balanced, then after switching it is all positive.  Then its incidence matrix is identical with that of $|\Sigma|$, which is well known to have rank $n-1$ (not $n$, because the rows sum to $0$; not less than $n$, because $\Eta(T)$ after dropping one row has determinant $\pm1$).  

If $2=0$ in $\bk$, then the definition shows that $\Eta(\Sigma) = \Eta(|\Sigma|)$ over $\bk$, and by the preceding argument the rank is $n-1$.

If $\Sigma$ is unbalanced, there is a negative edge $e$.  Consider the submatrix $\Eta(T\cup\{e\})$ of $\Eta(\Sigma)$ that consists of the columns corresponding to $T$ and $e$.  It has the form $\begin{pmatrix} \Eta(T) & \Eta(e) \end{pmatrix}.$  The last column has two $+1$'s or two $-1$'s, so its sum is $\pm2$.  The only linear dependence among the rows of $\Eta(T)$ is that they sum to $0$; but the sum of the rows of the matrix $\begin{pmatrix} \Eta(T) \ \Eta(e) \end{pmatrix}$ has $\pm2$ in its last column.  Therefore, if $2\neq0$ in $\bk$ the rows of $\begin{pmatrix} \Eta(T) \ \Eta(e) \end{pmatrix}$ are linearly independent; this submatrix has rank $n$.  Thus, $\Eta(\Sigma)$ has rank $n$.
\end{proof}

Since the incidence matrix of $|\Sigma|$ has rank $n-c(\Sigma)$, and $b(\Sigma) = c(\Sigma)$ when, and only when, $\Sigma$ is balanced, $\Eta(\Sigma)$ has the same rank as $\Eta(|\Sigma|)$ if and only if $\Sigma$ is balanced.

\subsection{The Kirchhoff Matrix and Matrix-Tree Theorems}\label{k}

The \emph{Kirchhoff matrix} (also called the \emph{Laplacian}) is 
\begin{equation}\label{E:kirchhoff}
K(\Sigma) := \Eta(\Sigma) \Eta(\Sigma)\transpose = D(|\Sigma|)-A(\Sigma),
\end{equation}
where $D(|\Sigma|)$ is the degree matrix of $|\Sigma|$, that is, the diagonal matrix whose diagonal entries are the degrees of the vertices in the underlying graph.\\

\begin{figure}[hb]
$
K(\Sigma_4) = \Eta(\Sigma_4,\eta) \Eta(\Sigma_4,\eta)\transpose = 
\begin{pmatrix}
3 & -1 & 1 & -1 \\
-1 & 2 & 1 & 0 \\
1 & 1 & 3 & -1 \\
-1 & 0 & -1 & 2
\end{pmatrix}
$
\caption{The Kirchoff matrix of $\Sigma_4$ from Figure \ref{F:sgor}.}
\label{F:sgkirchhoff}
\end{figure}

\begin{lem}\label{L:krank}
The rank of\/ $K$ (over the real numbers) is $n-b(\Sigma)$.  Its nullity is $b(\Sigma)$.
\end{lem}

\begin{proof}
A matrix of the form $MM\transpose$ is positive semidefinite and has rank equal to the rank of $M$.  By standard matrix theory, its eigenvalues are all non-negative.  In our case, $M = \Eta(\Sigma)$ has rank $n-b(\Sigma)$ by Theorem \ref{T:rank}.
\end{proof}

The classical Matrix-Tree Theorem expresses the number of spanning trees of a graph in terms of the Kirchhoff matrix: the number is $\det K_{ij}$ where $K_{ij}$ is $K$ with any one row and column deleted.  I state two signed-graphic analogs.  For proofs see the references; the first is not hard (it uses the Binet--Cauchy theorem in the standard way), but the second is rather complicated both to state fully and to prove.

\begin{thm} [Zaslavsky {\cite[Section 8A]{SG}}]\label{T:mt}
The determinant $\det K$ is the sum, over all pseudoforests\/ $F$ with $n$ edges and with no positive circles, of\/ $4^{c(F)}$, where $c(F)$ is the number of components of $F$.
\end{thm} 

The reason is that $\det \Eta(F) = 2^{c(F)}$.  Note that the pseudoforests in the theorem, because they have as many edges as vertices, can have no tree components.

\begin{thm}[Chaiken's All-Minors Matrix-Tree Theorem {\cite{AMMT}}]\label{T:ammt}
From $K$ delete $k$ rows, corresponding to vertices $r_1,\ldots, r_k$, and $k$ columns, corresponding to $c_1,\ldots,c_k$, and then take the determinant.  The resulting number is the sum of\/ $\pm4^q$, where $q$ is the number of circles in $F$, over all $n-k$-edge spanning pseudoforests without positive circles such that each tree component of\/ $F$ contains exactly one $r_i$ and one $c_j$.  The sign of the term is given by a complicated rule.
\end{thm}

Chaiken's theorem is even more general: the edges can be weighted.

\subsection{Eigenvalues of the Kirchhoff Matrix}\label{ke}

The eigenvalues of the Kirchhoff matrix are a natural topic of study.  As a simple example, they constrain the eigenvalues of $\Sigma$ when the underlying graph is regular, in just the same way as with ordinary graphs.

\begin{thm}\label{T:areg}
If the underlying graph of\/ $\Sigma$ is regular of degree $k$, then all eigenvalues of\/ $\Sigma$ are $\leq k$.  The value $k$ is an eigenvalue of multiplicity $b(\Sigma)$; thus, $k$ is an eigenvalue if and only if\/ $\Sigma$ has a balanced component.
\end{thm}

\begin{proof}
If $|\Sigma|$ is regular of degree $k$, then $D(|\Sigma|) = kI$.  Thus, $K = kI - A$.  The Kirchhoff matrix is positive semidefinite and has nullity $b(\Sigma)$ by Lemma \ref{L:krank}.  Therefore, all eigenvalues of $K$ are at least $0$.  The eigenvalues $\lambda_i^A$ of $A$ are determined by $\lambda_i^K = k - \lambda_i^A$ where $\lambda_1^K, \ldots, \lambda_n^K$ are the eigenvalues of $K$.  It follows that all $\lambda_i^A \leq k$ and that the eigenvalue $k$ of $A$ has the same multiplicity as the corresponding eigenvalue $0$ of $A$.
The multiplicity of $0$ as an eigenvalue of $K$ is the nullity of $K$.  That gives the multiplicity of $k$ as an eigenvalue of $A$.
\end{proof}

Hou, Li, and Pan \cite{HLP} studied the eigenvalues of $K$ in general, without assuming regularity of the underlying graph.  They have two kinds of results: upper and lower bounds on the largest eigenvalue of $K$ (mostly in terms of the underlying graph), and an interlacing theorem for all eigenvalues.  Write the eigenvalues in decreasing order as $\lambda^1 \geq \lambda^2 \geq \cdots \geq \lambda^n$.  Here are some of their results:

\begin{thm}[Hou, Li, and Pan \cite{HLP}] \label{T:kebounds}  
Let\/ $\Gamma$ be a connected simple graph and let\/ $\Sigma$ be a signed simple graph.\\[5pt]  
{\rm(1)}  $\lambda^1(\Gamma,\sigma) \leq \lambda^1(-\Gamma)$, with equality iff\/ $\sigma$ is antibalanced (Lemma 3.1).\\[5pt]
{\rm(2)}  $\lambda^1(\Sigma) \leq 2(n-1)$, with equality iff\/ $\Sigma$ switches to $-K_n$ (Theorem 3.4).\\[5pt]
{\rm(3)}  $\lambda_1(\Sigma^+) + \lambda_1(\Sigma^-) \geq \lambda_1(\Sigma) \geq \lambda_1(\Sigma^+), \lambda_1(\Sigma^-)$ (Corollary 3.8).\\[5pt]
{\rm(4)}  $\lambda_1(\Sigma) \geq 1 + \max_{v\in V} d_{|\Sigma|}(v)$ (Theorem 3.10).
\end{thm}

\begin{thm}[{Hou, Li, and Pan \cite[Lemma 3.7]{HLP}}]\label{T:keinterlace}
$\lambda^i(\Sigma) \geq \lambda^i(\Sigma \setminus e) \geq \lambda^{i+1}(\Sigma)$.
\end{thm}

\subsection{Examples}\label{ix}

\begin{outlineone}

\item An \emph{oriented incidence matrix} of the unsigned graph $\Gamma$ is the same as an incidence matrix $\Eta(+\Gamma)$ of the all-positive signed graph $+\Gamma$.  Since $+\Gamma$ is balanced, the rank given by our formula equals $n - c(\Gamma)$, as is well known. 

The Kirchhoff matrix of $+\Gamma$ is simply that of $\Gamma$, i.e., $D(\Gamma) - A(\Gamma)$.  Its determinant is zero, which is consistent with Theorem \ref{T:mt} because there are no negative circles in $+\Gamma$, so the number of pseudoforests of the kind counted by Theorem \ref{T:mt} is zero.

\smallskip
\item The \emph{unoriented incidence matrix of $\Gamma$} is a $(0,1)$-matrix which has $1$ in position $(v,e)$ if $v$ is an endpoint of $e$ and has $0$ otherwise.  (This matrix is often, though I think unfortunately, called simply `the incidence matrix' of $\Gamma$.)  
It is an incidence matrix of the all-negative signed graph, $-\Gamma$.  
Since an all-negative graph is balanced if and only if it is bipartite (Corollary \ref{C:bipartite}), the rank of the matrix (except in characteristic $2$) equals $n - b$ where $b$ is the number of bipartite components of $\Gamma$.  This result was previously obtained by ad hoc methods (originally by van Nuffelen \cite{vN}), but it is really a special case of the general rank theorem for signed graphs.
\label{X:unoriented}

The Kirchhoff matrix $K(-\Gamma)$ equals $D(\Gamma) + A(\Gamma)$.  Its determinant equals the sum of $4^{c(F)}$ over all $n$-edge pseudoforests $F$ of which no component is bipartite.

\smallskip
\item The signed graphs of Figure \ref{F:sgmulti} are connected and unbalanced.  Therefore, each of their incidence matrices has rank $|V|$, which is $4$.
\end{outlineone}

\section{Line Graphs}

\subsection{Unsigned Line Graphs}\ 

The line graph of an unsigned graph $\Gamma$ is denoted by $L(\Gamma)$.  Its vertex set is $E(\Gamma)$, and two edges are adjacent if they have a common endpoint in $\Gamma$.  $L(\Gamma)$ has two kinds of distinguished circles: \emph{vertex triangles} are formed by three edges incident with a common vertex, and \emph{derived circles} are the line graphs of circles in $\Gamma$.  Every circle in $L(\Gamma)$ is known to be a set sum (i.e., symmetric difference of edge sets) of vertex triangles and derived circles.

\subsection{Signed Line Graphs}\label{lgdef}

The line graph $\Lambda(\Sigma)$ of $\Sigma$ is a switching class, not a single signed graph \cite{LSD}.  Its underlying graph is the line graph $L(|\Sigma|)$ of the underlying graph.  To define $\Lambda(\Sigma)$ we may take the approach of edge orientation or a direct definition of the circle signs.

\begin{figure}[b]
\includegraphics[scale=0.75]{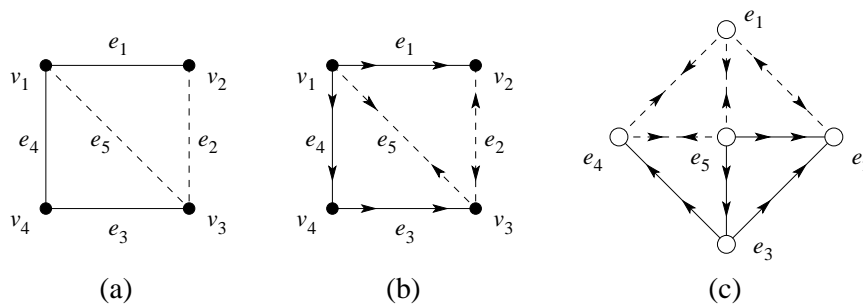}
\caption{A signed simple graph (a), an orientation (b), and the oriented line graph (c).}
\label{F:lg}
\end{figure}

\begin{outlineone}

\item \emph{Definition by Orientation \cite{LSD}.}
  
Choose an orientation $\eta$ of $\Sigma$.  We define a bidirection $\eta'$ of $L(|\Sigma|)$ and therefore an edge signature, thus forming the line signed graph $\Lambda(\Sigma)$.  Two $\Sigma$-edges $e_{vw}, e_{vu}$ incident with a vertex $v$ form an edge $e := e_{vw}e_{vu}$ in $\Lambda$, whose vertices are $e_{vw}$ and $e_{vu}$.  An end of $e$ may therefore be written $(e_{vw},e)$, corresponding to the end $(v,e_{vw})$ in $\Sigma$.  Define $\eta'(e_{vw},e) = \eta(v,e_{vw})$.

In terms of arrows, bidirect each edge of $\Sigma$ with two arrows as indicated by $\eta$, and let the arrow on $(e_{vw},e)$ point into the vertex $e_{vw}$ iff the arrow on $(v,e_{vw})$ points into the vertex $v$ in $\Sigma$.
 
Reorienting an edge in $\Sigma$ corresponds to switching the corresponding vertex in $\Lambda$.  Thus, $\Lambda(\Sigma)$ is well defined only up to switching.  I.e., it is a well defined switching class.

\smallskip
\item \emph{Definition by Circle Signs \cite{LGSC, LSD}.}   

Make every vertex triangle negative and give to every derived circle the same sign as the circle in $\Sigma$ it derives from.  Other circles get signed by the following \emph{sum rule}:  If $C$ is the set sum of certain vertex triangles and derived circles, its sign is the product of the signs of those vertex triangles and derived circles.  This rule is a consequence of Proposition \ref{P:circles}.   One has to prove that the sum rule gives the same sign no matter how it is applied, which is most easily achieved by showing that the definition by circle signs agrees with that by edge orientation, the latter being obviously well defined.
\end{outlineone}

\subsection{Reduced Line Graphs}\label{lgred}

If we allow $\Sigma$ to have parallel edges---though only with opposite sign, so that $\Sigma$ is simply signed but two vertices can be joined by both a positive and a negative edge---there is a definition of line graph similar to the preceding one.  In the line graph there are negative digons derived from those of $\Sigma$.  In the line-graph adjacency matrix $A(\Lambda)$, these digon edges cancel.  To represent that phenomenon accurately in the line graph we should \emph{reduce} the line graph by deleting pairs of parallel edges, $+ef$ and $-ef$, of opposite sign.  The \emph{reduced line graph}, $\bar\Lambda(\Sigma)$, is what results.  It is a signed simple graph.

The importance of reduction is that, even though the underlying graph of $\Sigma$ itself may not be simple, that of $\bar\Lambda(\Sigma)$ is.  
More precisely, if the underlying graph $|\Sigma|$ is simple, $|\Lambda(\Sigma)|$ is already simple and does not need to be reduced.  If $|\Sigma|$ is not simple, then $|\Lambda(\Sigma)|$ is not simple but the underlying graph of the reduced line graph is simple.  
Thus, amongst the signed simple graphs, there are graphs that are reduced line graphs of simply signed graphs but not line graphs of simply signed graphs.  One is shown in Figure \ref{F:reducedlg}; since $|\bar\Lambda(\Sigma_5)|$ is not a line graph, $\bar\Lambda(\Sigma_5)$ cannot be an unreduced line graph.  

\begin{figure}[ht]
\includegraphics[scale=0.75]{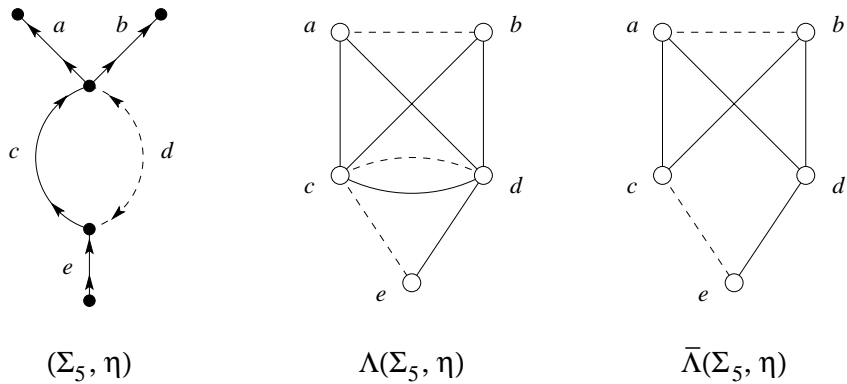}
\caption{An oriented signed graph $(\Sigma_5,\eta)$, its line graph $\Lambda(\Sigma_5)$ and its reduced line graph $\bar\Lambda(\Sigma_5)$, whose underlying graph $|\bar\Lambda(\Sigma_5)|$ is not a line graph.}
\label{F:reducedlg}
\end{figure}
%

\subsection{Geometry}\label{lggeom}

This larger class, the reduced line graphs, has been studied by Vijayakumar and his coworkers, under the name of `signed graphs represented by $D_\infty$' \cite{Vd, CV}.  
From our line-graphic point of view, the name comes from the fact that 
$$
D_n := \{ \pm \bb_i \pm \bb_j : 0 < i < j \leq n \},
$$
where $\{\bb_1, \ldots, \bb_n\}$ is the standard orthonormal basis of $\bbR^n$, is the set of all those vectors that can be the column of an edge in the incidence matrix of a signed graph with $n$ vertices; the line graph corresponds to taking dot products of these vectors.  Allowing $n$ to be arbitrarily large, one has $D_\infty$.  

$D_n$ is best known as one of the classical root systems in Lie algebra.  The other important root system for us is $E_8$, an exceptional root system in $\bbR^8$ whose exact definition is unnecessary here---but see \cite{Ve8, Ve8a}.

We need to define representation by $W \subseteq \bbR^n$ since it depends on the adjacency matrix and is at the heart of the treatment of line graphs.  A signed simple graph $\Sigma$ is \emph{represented by $W$} if there is an injection $f: V \to W$ such that the dot product $f(v) \cdot f(w) = \sigma(e_{vw})$ if $v$ and $w$ are adjacent and $0$ if they are not, and all $\|f(v)\| = \sqrt 2$.  (That ought to have been $-\sigma(e_{vw})$ to be consistent with our definition of line graph but the negative sign was omitted by Vijayakumar, who defined this terminology.)

\subsection{Adjacency Matrix and Eigenvalues}\label{lgae}

The adjacency matrix of $\Lambda(\Sigma)$ is the $E \times E$ matrix given by 
\begin{equation}\label{E:lga}
A(\Lambda(\Sigma)) = 2I - \Eta(\Sigma)\transpose \Eta(\Sigma).
\end{equation}

\begin{thm}\label{T:lgae}
The largest eigenvalue of\/ $A(\Lambda(\Sigma))$ is at most $2$.  Moreover, $2$ is an eigenvalue of\/ $A(\Lambda(\Sigma))$ with multiplicity $|E| - n + b(\Sigma)$; in particular, it is an eigenvalue if and only if\/ $\Sigma$ has a connected component that is neither a tree nor an unbalanced 1-tree.
\end{thm}

\begin{proof}
A matrix product of the form $M\transpose M$ is positive semidefinite, so all its eigenvalues are at least $0$.  An eigenvalue $\beta$ of $M\transpose M$ corresponds to the eigenvalue $\alpha-\beta$ of $\alpha I - M\transpose M$.  Letting $M = \Eta$ and $\alpha=2$, we deduce that the eigenvalues of the right side of Equation \eqref{E:lga} are not greater than $2$.

By matrix theory, the rank of $\Eta\transpose \Eta$ equals that of $\Eta$, which is $n - b(\Sigma)$.  It has $0$ as an eigenvalue of multiplicity (order $-$ rank).  The order of $\Eta\transpose \Eta$ is $|E|$.  Thus, it has $0$ as an eigenvalue of multiplicity $|E| - n + b(\sigma)$.  The corresponding eigenvalue $2$ of $A(\Lambda(\Sigma))$ has the same multiplicity.
\end{proof}

\begin{thm}[{\cite{CGSS}, \cite[Theorem 1.1]{CV}}]\label{T:e2}
If $\Sigma$ is a signed simple graph whose eigenvalues are $\leq 2$, then $-\Sigma$ is represented by $D_n$ or $E_8$.
\end{thm}

That is, the connected signed simple graphs whose eigenvalues are at most $2$ are the reduced line graphs of simply signed graphs (represented by $D_n$) and only a few additional sporadic examples (represented by $E_8$) which are neither line graphs nor reduced line graphs.  This fact was explicitly recognised by Chawathe and Vijayakumar in \cite[Theorem 1.1]{CV} (or see \cite[the first Theorem 2.4, on p.~214]{VS}), although the proof is actually in the classic paper of Cameron, Goethals, Seidel, and Shult \cite{CGSS}.  

The essential observation behind Theorem \ref{T:e2} is that, if $\Sigma$ has all eigenvalues $\leq 2$, then $2I - A(\Sigma)$ is a positive semidefinite matrix and therefore is the matrix of inner products of a set $W$ of vectors in $\bbR^m$ for some $m$.  The proof involves classifying the possible sets $W$, which turn out to be the subsets of $D_n$ and $E_8$.

Singhi and Vijayakumar showed that having an eigenvalue $>2$ implies an induced subgraph with all eigenvalues $\leq2$; indeed, their result is stronger:

\begin{thm}[Singhi and Vijayakumar \cite{SV}]\label{T:egt2}
If a signed simple graph $\Sigma$ has an eigenvalue greater than or equal to $2$, then it contains an induced subgraph whose largest eigenvalue is exactly $2$.
\end{thm}

This result is the harder converse of the following corollary of Proposition \ref{P:einterlacing}.

\begin{prop}\label{P:egt2}
If $\Sigma$ has an induced subgraph whose largest eigenvalue is $2$, then the largest eigenvalue of $\Sigma$ is at least $2$.
\end{prop}

\subsection{Examples}\label{lgx}

\begin{outlineone}
\item The line graph of the all-negative signed graph $-\Gamma$ is $[-L(\Gamma)]$, the switching class of the ordinary line graph with all negative signs.  From the standpoint of line graphs, therefore, ordinary graphs should be considered as all-negative signed graphs, instead of all-positive as in other parts of signed graph theory.

\smallskip
\item Hoffman's generalised line graph \cite{GLG} is the reduced line graph of a signed graph; specifically, of $-\Gamma(m_1,\ldots,m_n)$, which consists of $-\Gamma$ together with $m_i$ negative digons attached at the vertex $v_i$.  See Figure \ref{F:glg}.
\label{X:glg}
\begin{figure}[ht]
\includegraphics[scale=0.8]{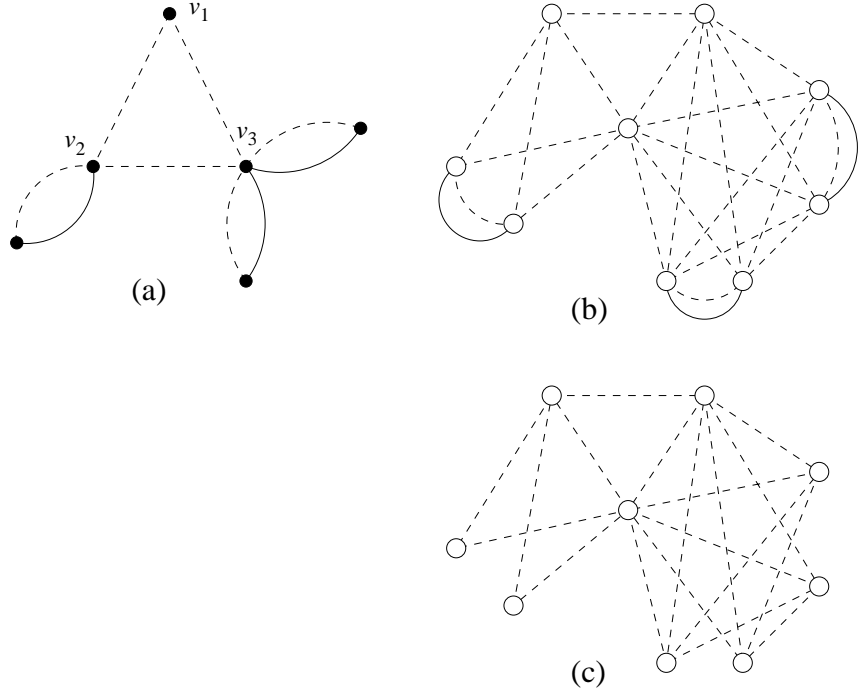}
\caption{A signed graph $-K_3(0,1,2)$ and its reduced line graph, which is a generalised line graph with all negative signs.}
\label{F:glg}
\end{figure}

\smallskip
\item The line graphs $\Lambda(\Sigma)$ that are antibalanced are those of the form $[-\Gamma']$ where $\Gamma_0$ is an ordinary line graph or a generalised line graph; that is:  

\begin{thm}[Cameron, Goethals, Seidel, and Shult \cite{CGSS}]\label{T:eglg}
The all-negative signed graphs that are reduced line graphs of signed graphs are precisely the all-negative generalised line graphs $-\Gamma(m_1,\ldots,m_n)$.
\end{thm}

\noindent 
(See the statement in the introduction to \cite{Vd} [where `The family of sigraphs represented by $D_\infty$' should be `The family of \emph{graphs} \dots'].)  
Again, we see that the usual line graphs are all-negative signed graphs.
(This result was not originally stated explicitly in terms of signed graphs, but in an equivalent fashion in terms of adjacency matrices.)  

\smallskip
\item As a directed graph is a bidirected all-positive graph $(+\Gamma,\eta)$, it has a signed line graph.  The positive part $\Lambda^+(+\Gamma,\eta)$ is a directed graph as well; it is precisely the Harary--Norman line digraph of \cite{LDG}.
\end{outlineone}

\subsection{History}\label{lghist}

The two definitions of a line graph of a signed graph are from \cite{LSD}, which has been on the verge of being written for more than two decades.

The interpretation of graphs represented by $D_\infty$, including Hoffman's generalised line graphs (Example \ref{lgx}.\ref{X:glg}), as line graphs was first stated in \cite[Example 2]{LGSC}.  It is implicit in the geometrical representation in \cite{CGSS}, in Vijayakumar's geometrical representation by $D_n$, and also in Godsil and Royle's presentation \cite[Section 12.1]{AGT}, based on the seminal paper \cite{CGSS}, which talks of cancelling double edges but without recognising the cancellation as due to opposite signs.



\begin{thebibliography}{99}
\bibitem{WSR} ------, 
What is a strongly regular signed graph?
In preparation.
\end{thebibliography}
\end{document}